\documentclass[12pt,draft]{article} 
\usepackage{amsmath,amssymb,amsthm,amsfonts,bm}
\usepackage{enumerate,color}

\makeatletter
\def\@cite#1#2{[{{\bfseries #1}\if@tempswa , #2\fi}]}
\renewcommand{\section}{%
\@startsection{section}{1}{\z@}
{0.5truecm plus -1ex minus -.2ex}%
{1.0ex plus .2ex}{\bfseries\large}}
\def\@seccntformat#1{\csname the#1\endcsname.\ }
\makeatother

\numberwithin{equation}{section} 
\newtheorem{thm}{Theorem}[section]
\newtheorem{corollary}[thm]{Corollary}
\newtheorem{lem}[thm]{Lemma}

\theoremstyle{definition}

\newtheorem{remark}{Remark}[section]

\newcommand{\ep}{\varepsilon}
\newcommand{\pa}{\partial}
\newcommand{\Rn}{\mathbb{R}^n}

\newcommand{\vtheta}{\vartheta}

\newcommand{\ol}[1]{\overline{#1}}

\newcommand{\tmax}{T_{\rm max,\lambda}}
\newcommand{\lp}[2]{\left\Vert{#2}\right\Vert_{L^{#1}(\Omega)}}

\newcommand{\cd}{(\cdot,t)}
\newcommand{\into}{\int_\Omega}
\newcommand{\obar}{\overline{\Omega}}
\newcommand{\utau}{u_\lambda}
\newcommand{\vtau}{v_\lambda}
\newcommand{\lambdan}{\lambda_n}
\newcommand{\lambdanj}{\lambda_{n_j}}
\newcommand{\utaunj}{u_{\lambda_{n_j}}}
\newcommand{\vtaunj}{v_{\lambda_{n_j}}}
\newcommand{\vphi}{\varphi_{r}}
\newcommand{\lambdaz}{\lambda_0}
\newcommand{\itau}{\lambda \in (0,\lambdaz)}
\newcommand{\uinit}{u_{\rm init}}
\newcommand{\vinit}{v_{\rm init}}
\newcommand{\uobar}{\overline{u}}
\newcommand{\vobar}{\overline{v}}

\begin{document}
\footnote[0]
    {2010{\it Mathematics Subject Classification}\/. 
    Primary: 35A09; Secondary:  35K51, 35J15, 92C17.
    }
\footnote[0]
    {{\it Key words and phrases}\/: 
     chemotaxis system with strong signal sensitivity; parabolic-parabolic system; parabolic-elliptic system.  
    }
\begin{center}
    \Large{{\bf 
    The fast signal diffusion limit in a chemotaxis system with strong signal sensitivity
    }}
\end{center}
\vspace{5pt}
\begin{center}
    Masaaki Mizukami\footnote{Partially supported by 
    JSPS Research Fellowships for Young Scientists (No. 17J00101).} 
   \footnote[0]{
    E-mail: 
    {\tt masaaki.mizukami.math@gmail.com} 
    }\\
    \vspace{12pt}
    Department of Mathematics, 
    Tokyo University of Science\\
    1-3, Kagurazaka, Shinjuku-ku, 
    Tokyo 162-8601, Japan\\
    \vspace{2pt}
\end{center}
\begin{center}    
    \small \today
\end{center}
\vspace{2pt}
\newenvironment{summary}
{\vspace{.5\baselineskip}\begin{list}{}{%
     \setlength{\baselineskip}{0.85\baselineskip}
     \setlength{\topsep}{0pt}
     \setlength{\leftmargin}{12mm}
     \setlength{\rightmargin}{12mm}
     \setlength{\listparindent}{0mm}
     \setlength{\itemindent}{\listparindent}
     \setlength{\parsep}{0pt}
     \item\relax}}{\end{list}\vspace{.5\baselineskip}}
\begin{summary}
{\footnotesize {\bf Abstract.}
This paper gives a first insight into making a mathematical bridge 
between the parabolic-parabolic signal-dependent chemotaxis system 
and its parabolic-elliptic version. 
To be more precise, this paper deals with 
convergence of a solution for 
the parabolic-parabolic chemotaxis system with 
strong signal sensitivity 
\begin{align*}
  (\utau)_t = \Delta \utau - \nabla \cdot (\utau \chi(\vtau)\nabla \vtau), 
\quad 
  \lambda (\vtau)_t = \Delta \vtau - \vtau +\utau
  \quad  
 \mbox{in} \ \Omega\times (0,\infty)
\end{align*}
to that for the parabolic-elliptic chemotaxis system 
\begin{align*}
 u_t = \Delta u -\nabla \cdot (u\chi(v)\nabla v), 
 \quad 
 0= \Delta v -v +u 
\quad  
 \mbox{in} \ \Omega\times (0,\infty), 
\end{align*}
where $\Omega$ is a bounded domain in $\mathbb{R}^n$ ($n\in\mathbb{N}$) with smooth boundary, $\lambda>0$ is a constant and $\chi$ is a function generalizing 
\[ 
  \chi(v) = \frac{\chi_0}{(1+v)^k} 
  \quad (\chi_0>0,\ k>1).
\] 
In chemotaxis systems 
parabolic-elliptic systems often provided some guide to 
methods and results for 
parabolic-parabolic systems. 
However,  the relation between parabolic-elliptic systems and parabolic-parabolic systems 
has not been studied. 
Namely, it still remains to analyze on the following question: 
{\it Does a solution of the parabolic-parabolic system converge to 
that of the parabolic-elliptic system as $\lambda \searrow 0$?} 
This paper gives some positive answer in 
the chemotaxis system with strong signal  sensitivity. 
}
\end{summary}
\vspace{10pt} 

\newpage
%
%

\section{Introduction}

{\it Partial differential equation} is one of topics of mathematical analysis, 
and many mathematicians study variations of partial differential equations intensively.  
Here there are several types of partial differential equations, 
e.g., {\it parabolic partial differential equation} 
and {\it elliptic partial differential equation}, 
and these often describe many phenomena 
which appear in natural science, especially, physics, chemistry and biology. 
Therefore partial differential equations frequently play an important role 
in analysis of some phenomenon, 
moreover, in some case elliptic partial differential equations help 
analysis of parabolic partial differential equations, 
e.g., analysis of steady states or simplified equations of 
parabolic partial differential equations. 
Here we focus on biological phenomena, especially {\it chemotaxis} 
which is one of important properties and is related to e.g., 
the movement of sperm, the migration of neurons or lymphocytes and the tumor invasion. 
Chemotaxis is the property such that species 
move towards higher concentration of a chemical substance 
when they plunge into hunger. 
One of examples of species 
which have chemotaxis is {\it Dictyostelium discoideum}. 
Keller--Segel \cite{K-S} studied the aggregation of {\it Dictyostelium discoideum} 
due to an attractive chemical substance, 
and proposed the following system of partial differential equations 
\begin{align*}
  u_t = \Delta u - \chi \nabla \cdot (u\nabla v), 
\quad  
  \lambda v_t = \Delta v -v +u 
\end{align*}
in $\Omega\times(0,\infty)$,  
where 
$\Omega\subset \mathbb{R}^n$ ($n\in \mathbb{N}$) is a bounded domain, 
$\chi>0$, $\lambda=0$ ({\it parabolic-elliptic system}) or 
$\lambda>0$ ({\it parabolic-parabolic system}). 
The above system is called as {\it Keller--Segel system} or {\it chemotaxis system}. 
In this problem the parabolic-elliptic system was first investigated, 
and then the result for the parabolic-elliptic system provided us some conjecture 
of results for the parabolic-parabolic system  
and the interaction between the parabolic-elliptic system 
and the parabolic-parabolic system 
made progress on researches of the Keller--Segel system. 
In the two-dimensional parabolic-elliptic case 
Nagai \cite{Nagai_1995} showed a condition for 
existence of global bounded solutions or blow-up solutions 
in the radially symmetric situation, which tells us that 
the size of initial data will determine whether classical solutions of the above system 
exist globally or not. 
After the above pioneering work,  the Keller--Segel system was studied intensively; 
conditions for global existence or blow-up 
in the parabolic-elliptic system were studied in 
\cite{Nagai_2001,Nagai-Senba-Yoshida} and in the parabolic-parabolic system 
were investigated in 
\cite{Biler_1998,Xinru_higher,Gajewski-Zacharias_1998,Herrero-Velazques_1997,Horstmann-Wang,Nagai-Senba-Yoshida,win_aggregationvs}; 
blow-up asymptotics of solutions 
for the parabolic-elliptic system is in \cite{Senba_2007,Senba-Suzuki_2001}
and for the parabolic-parabolic system is in 
\cite{Herrero-Velazques_1997,Mizoguchi-Souplet_2014,Nagai-Senba-Suzuki_2000}. 
More related works can be found in \cite{Bai-Winkler_2016,Black-Lankeit-Mizukami_01,Lin-Mu-Wang,Mizukami_DCDSB,Mizukami_PPEpro,stinner_tello_winkler,Tello-Winkler_2007,Tello_Winkler_2012,Winkler_2010_logistic}; 
a chemotaxis system with logistic term 
in the parabolic-elliptic case is in \cite{Tello-Winkler_2007} and  
in the parabolic-parabolic case is in \cite{Winkler_2010_logistic}; 
global existence and stabilization in a two-species chemotaxis-competition system 
were shown in the parabolic-parabolic-elliptic case 
(\cite{Black-Lankeit-Mizukami_01,Mizukami_PPEpro,
stinner_tello_winkler,Tello_Winkler_2012}) and 
in the parabolic-parabolic-parabolic case 
(\cite{Bai-Winkler_2016,Lin-Mu-Wang,Mizukami_DCDSB}). 

Moreover, the chemotaxis system with signal-dependent sensitivity 
\begin{align*}
  u_t = \Delta u -\nabla \cdot (u \chi(v)\nabla v), 
  \quad 
  \lambda v_t = \Delta v -v+u  
\end{align*}
in $\Omega\times(0,\infty)$,  
where 
$\Omega\subset \mathbb{R}^n$ ($n\in \mathbb{N}$) is a bounded domain, 
$\lambda\ge 0$ is a constant and 
$\chi$ is a nonnegative function generalizing 
\begin{align*}
  \chi(v)= \frac{\chi_0}{v}
 \quad \mbox{and} \quad 
\chi(v)=\frac{\chi_0}{(1+v)^k}\quad  (\chi_0>0,\ k>1), 
\end{align*}
was also studied in the parabolic-elliptic case firstly, 
and then researches of the above system were developed by the interaction 
between the parabolic-elliptic system and the parabolic-parabolic system. 
In the parabolic-elliptic system with $\chi(v)=\frac{\chi_0}{v}$ ($\chi_0>0$) 
Nagai--Senba \cite{Nagai-Senba_1998} showed that 
if $n=2$, or $n\ge 3$ and $\chi_0<\frac{2}{n-2}$ then a radial solution is global and bounded, 
and if $n\ge 3$ and $\chi_0>\frac{2n}{n-2}$ then there exists some initial data such that 
a radial solution blows up in finite time. 
In the nonradial case Biler \cite{Biler_1999} obtained global existence of solutions to 
the parabolic-elliptic system with $\chi(v)=\frac{\chi_0}{v}$ ($\chi_0>0$) 
under the conditions that $n=2$ and $\chi_0\le 1$, or $n\ge 3$ and $\chi_0<\frac{2}{n}$.  
Thanks to these results, we can expect that conditions for global existence in 
the above system were 
determined by a dimension of a domain and  
a smallness of $\chi$ in some sense. 
Indeed, global existence and boundedness of 
solutions to the parabolic-elliptic system 
with $\chi(v)=\frac{\chi_0}{v^k}$ ($\chi_0>0$, $k\ge 1$) 
were derived 
under some smallness conditions for $\chi_0$ (\cite{Fujie-Winkler-Yokota_pe}). 
Moreover, Fujie--Senba \cite{F-S;p-e} established global existence and boundedness 
in the two-dimensional parabolic-elliptic system with more general sensitivity function. 
On the other hand, also in the parabolic-parabolic case, it was shown that 
some smallness condition for $\chi$ implies global existence and boundedness; 
in the case that $\chi(v)=\frac{\chi_0}{v}$ ($\chi_0>0$) 
Winkler \cite{Winkler_2011} obtained global existence of classical solutions 
under the condition that $\chi_0<\sqrt{\frac{2}{n}}$ 
and Fujie \cite{Fujie_2015} established boundedness of these solutions;  
in the case that $\chi(v)\le \frac{\chi_0}{(a+v)^k}$ ($\chi_0>0$, $a\ge0$, $k>1$) 
some smallness condition for $\chi_0$ leads to 
global existence and boundedness (\cite{Mizukami-Yokota_02}); 
recently, Fujie--Senba \cite{F-S;p-p} showed global existence and boundedness of 
radially symmetric solutions to the parabolic-parabolic system 
with more general sensitivity function and small $\lambda$ 
in a two-dimensional ball. 

In summary, in the setting that $\Omega$ is a bounded domain, 
parabolic-elliptic chemotaxis systems 
often provided us some guide to 
how we could deal with parabolic-parabolic chemotaxis systems;  
however, the relation between the both systems has not been studied. 
Namely, in the setting that $\Omega$ is a bounded domain, 
it still remains to analyze on the following question: 
\begin{center}
{\it Does a solution of the parabolic-parabolic system converge to\\ 
that of the parabolic-elliptic problem as $\lambda \searrow 0$?} 
\end{center}
If we can obtain some positive answer to this question, 
then we can see that solutions of both systems have some similar properties; 
thus an answer will enable us to establish 
approaches to obtain properties 
for solutions of the chemotaxis systems.  %
Here, in the setting that $\Omega$ is the whole space $\mathbb{R}^n$, 
there are some positive answers to this question in $2$-dimensional case (\cite{Raczynski_2009}) 
and $n$-dimensional case (\cite{Lemarie_2013}). 
Therefore we can expect a positive answer to this question also in the setting that 
$\Omega$ is a bounded domain.  
In order to obtain an answer to this question 
we first deal with the chemotaxis system with strong signal sensitivity, 
because we have already provided all tools to establish  
the $L^\infty (\Omega \times [0,\infty))$-estimate for solutions via a priori estimate 
in \cite{Mizukami-Yokota_02} 
(This is likely to enable us to see a uniform-in-$\lambda$ estimate). 
The purpose of this work is to obtain some positive answer to this question 
in the chemotaxis system with strong signal sensitivity.  

%
%
In this paper we consider convergence of a solution for 
the parabolic-parabolic chemotaxis system with 
signal-dependent sensitivity 
\begin{equation}\label{cp}
  \begin{cases}
    (\utau)_t = \Delta \utau 
    - \nabla \cdot (\utau\chi(\vtau) \nabla \vtau),
    & x\in\Omega,\ t>0, 
\\[1mm]
    \lambda (\vtau)_t = \Delta \vtau - \vtau + \utau, 
    & x\in\Omega,\ t>0, 
\\[1mm]
    \nabla \utau\cdot \nu = 
    \nabla \vtau\cdot \nu = 0, 
    & x\in\pa \Omega,\ t>0,
\\[1mm]
    \utau(x,0)=\uinit(x),\ \vtau(x,0)=\vinit(x),
    & x\in\Omega
  \end{cases}
\end{equation}
to that of the parabolic-elliptic chemotaxis system  
\begin{equation}\label{cpp-e}
  \begin{cases}
    u_t = \Delta u 
    - \nabla \cdot (u\chi(v) \nabla v),
    & x\in\Omega,\ t>0, 
\\[1mm]
    0 = \Delta v - v + u, 
    & x\in\Omega,\ t>0, 
\\[1mm]
    \nabla u \cdot \nu = 
    \nabla v\cdot \nu = 0, 
    & x\in\pa \Omega,\ t>0,
\\[1mm]
    u(x,0)=\uinit(x),
    & x\in\Omega, 
  \end{cases}
\end{equation}
where $\Omega$ is a bounded domain in $\Rn$ 
($n\ge 2$) 
with smooth boundary $\pa \Omega$ 
and $\nu$ is 
the 
outward normal vector to $\pa\Omega$; 
$\lambda > 0$ is a constant; 
the initial functions $\uinit,\vinit$ are 
assumed to be nonnegative functions; 
the sensitivity function $\chi$ is assumed to be 
generalization of the regular function: 
\[
  \chi(s)=\frac{\chi_0}{(1+s)^k} \quad 
  (s>0),
\]
where $\chi_0>0$ and $k>1$ are constants. 
The unknown functions $\utau (x,t)$ and $u(x,t)$   
represent the population density of the species and 
$\vtau (x,t)$ and $v(x,t)$ show the concentration of the 
chemical substance 
at place $x$ and time $t$. 

%
%
%
%
Now the main results 
read as follows. 
We suppose that the sensitivity function $\chi$ 
satisfies that 
\begin{align}\label{mainass}
\chi\in C^{1+\vtheta}(0,\infty)\quad \mbox{and}\quad 
0\le \chi(s) \le \frac{\chi_0}{(a+s)^k} 
\quad (s>0)
\end{align}
with some $\vtheta\in (0,1)$, $a\ge 0$, $k>1$ 
and $\chi_0>0$. 
The first theorem is concerned with 
global existence and boundedness 
in \eqref{cp} under a condition depending on $\lambda$. 
\begin{thm}\label{mainth} 
Let $\Omega$ be a bounded domain in $\Rn$ 
$(n\ge 2)$ 
with smooth boundary, 
and let $\lambda > 0$. 
Assume that $\chi$ satisfies \eqref{mainass} 
with some $\vtheta\in (0,1)$, $a\ge0$, $k>1$ and $\chi_0>0$ satisfying  
\begin{align}\label{mainass2}
  \chi_0 < \frac{4k(a+\eta)^{k-1}}
  {(1-\lambda)_{+}n + \sqrt{n(n\lambda^2-2n\lambda +n + 8\lambda)}}, 
\end{align}
where 
\begin{align*}
  \eta:=
  \sup_{\tau>0}\left(
  \min\left\{
  e^{-2\tau} \min_{x\in\overline{\Omega}}\vinit (x),\ 
  c_0\lp{1}{\uinit}
  (1-e^{-\tau})
  \right\}
  \right)
\end{align*}
and a constant $c_0>0$ 
is a lower bound for the fundamental solution of 
$w_t = \Delta w -w$ with Neumann boundary condition. 
Then for all 
$\uinit, \vinit$ satisfying 
\begin{align}\label{ini} 
0\leq \uinit\in C(\ol{\Omega})\setminus \{0\}, \quad
\begin{cases}
0< \vinit\in W^{1,q}(\Omega) \ (\exists \, q>n) 
& (a=0),\\
0\leq \vinit \in W^{1,q}(\Omega)\setminus \{0\}\ 
(\exists \, q>n) 
& (a \neq 0), 
\end{cases}
\end{align} 
the problem \eqref{cp} possesses 
a unique global solution 
\[
  \utau,\vtau \in 
  C(\overline{\Omega}\times [0,\infty))
  \cap 
  C^{2,1}(\overline{\Omega}\times (0,\infty))
\]
satisfying that there exists $C>0$ such that 
\[
  \lp{\infty}{\utau\cd}+\|\vtau\cd\|_{W^{1,q}(\Omega)}\le C
\]
for all $t>0$. 
\end{thm}
%
%

The next corollary gives existence of global solutions 
satisfying a uniform-in-$\lambda$ estimate under a condition independent of $\lambda$. 

\begin{corollary}\label{mainth2}
Let $\Omega$ be a bounded domain in $\Rn$ 
$(n\ge 2)$ 
with smooth boundary, 
and assume that $\chi$ satisfies \eqref{mainass} 
with some $\vtheta\in (0,1)$, $a\ge0$, $k>1$ and $\chi_0>0$ satisfying  
\begin{align}\label{mainass3}
  \chi_0 < \frac{2k(a+\eta)^{k-1}}{n}. 
\end{align}
Then for all $\uinit,\vinit$ satisfying \eqref{ini}, 
there exists $\lambda_0\in (0,1)$ such that 
for all $\lambda\in (0,\lambda_0)$ 
the problem \eqref{cp} possesses a unique global classical solution 
$(\utau,\vtau)$ satisfying 
that there exists $C>0$ independent of $\lambda\in (0,\lambda_0)$ such that 
\begin{align}\label{estimate;uniformintau}
  \lp{\infty}{\utau\cd}+\|\vtau\cd\|_{W^{1,q}(\Omega)}\le C
\end{align}
for all $t>0$ and any $\lambda \in (0,\lambda_0)$. 
\end{corollary}

Then the uniform-in-$\lambda$ estimate for the solution 
obtained in Corollary \ref{mainth2} 
leads to the following result. 

\begin{thm}\label{mainth3}
Let $\Omega$ be a bounded domain in $\Rn$ 
$(n\ge 2)$ 
with smooth boundary, 
and assume that $\chi$ satisfies \eqref{mainass} 
and \eqref{mainass3}. 
Then for all $\uinit,\vinit$ satisfying \eqref{ini}, 
there exist unique functions 
\[
  u\in C(\obar\times [0,\infty))\cap 
  C^{2,1}(\obar\times (0,\infty)) 
\ \mbox{and} \ 
  v \in C^{2,0}(\obar\times (0,\infty))\cap 
  L^\infty([0,\infty);W^{1,q}(\Omega)) 
\]
such that the solution $(\utau,\vtau)$ of \eqref{cp} satisfies 
\begin{align*}
 &\utau \to u \quad 
 \mbox{in}\ C_{\rm loc}(\obar\times [0,\infty)),
\\
 &\vtau \to v \quad \mbox{in}\ 
 C_{\rm loc}(\obar\times (0,\infty)) \cap L^2_{\rm loc}((0,\infty);W^{1,2}(\Omega)) 
\end{align*}
as $\lambda\searrow 0$. 
Moreover, the pair of the functions $(u,v)$ solves \eqref{cpp-e} classically. 
\end{thm}

Difficulties are caused by the facts that 
$\vtau$ satisfies a parabolic equation and 
$v$ satisfies an elliptic equation.  
Thus we cannot use methods only for parabolic equations and 
only for elliptic equations when we would like to obtain some 
error estimate for solutions of \eqref{cp} and those of \eqref{cpp-e}, 
and it seems to be difficult to combine these methods. 
Therefore we rely on a compactness method to obtain 
convergence of a solution $(\utau,\vtau)$ as $\lambda\searrow 0$. 
In order to use a compactness method 
some uniform-in-$\lambda$ estimate for the solution is required. 
The strategy of seeing an estimate independent of $\lambda$  
is to modify the methods in 
\cite{Mizukami-Yokota_02}. 
One of keys for this strategy is to derive the differential inequality 
\[
  \frac d{dt} \into \utau^p \varphi (\vtau) 
  \le -c_1 \left( \into \utau^p \varphi (\vtau)\right)^b 
  + c_2 \into \utau^p \varphi (\vtau) 
  +c_3, 
\]
where $\varphi$ is some function, and $b>1$ and $c_1,c_2,c_3>0$ are some constants. 
This together with a smoothing property of $(e^{\tau\Delta})_{\tau\ge 0}$ 
enables us to establish the desired estimate. 
Then we can see convergence of the solutions $(\utau,\vtau)$ as $\lambda\searrow 0$ 
by using the uniform-in-$\lambda$ estimate and the Arzel\`{a}--Ascoli theorem. 

%
%

This paper is organized as follows. 
In Section 2 we collect basic facts 
which will be used later. 
In Section 3 
we prove global existence and 
boundedness in \eqref{cp} for all $\lambda>0$, 
and establish a uniform-in-$\lambda$ estimate 
(Theorem \ref{mainth} and Corollary \ref{mainth2}).  
Section 4 is devoted to the proof of Theorem \ref{mainth3} 
according to arguments in \cite{Wang-Winkler-Xiang}; 
we show convergence of the solution $(\utau,\vtau)$ for \eqref{cp} 
as $\lambda\searrow 0$ 
by using the uniform-in-$\lambda$ estimate 
established in Corollary \ref{mainth2}. 

%
%

\section{Preliminaries}

In this section we collect results which will be used later. 
We first introduce the uniform-in-time 
lower estimate for $\vtau$ which is independent of 
$\lambda>0$. 

%
%
%
%
\begin{lem}\label{loweresti}
Let $\lambda>0$ and let $u\in C(\overline{\Omega}\times [0,T))$ be 
a nonnegative function such that, with some $m>0$, 
$\int_\Omega u(\cdot,t)
=m$ for every $t\in[0,T)$. 
If  $\vinit\in C(\overline{\Omega})$ is 
a nonnegative function 
in $\overline{\Omega}$ and 
$\vtau\in C^{2,1}(\overline{\Omega}\times (0,T)) 
\cap C(\overline{\Omega}\times [0,T))$ 
is a classical solution of 
\begin{align*}
\begin{cases}
  \lambda (\vtau)_t=\Delta \vtau - \vtau + u, 
&
  x\in\Omega,\ t\in (0,T),
\\
  \nabla \vtau\cdot \nu =0, 
&
  x\in\partial\Omega,\ t\in(0,T), 
\\
  \vtau(x,0)=\vinit (x), 
&
  x\in \Omega, 
\end{cases}
\end{align*}
then  
\begin{align*}
\inf_{x\in \Omega}\vtau(x,t)\geq \widetilde{\eta} 
\end{align*}
holds for all $t\in(0,T)$, where 
\begin{align}\label{defetatilde}
  \widetilde{\eta}:=
  \sup_{\tau>0}\left(
  \min\left\{
  e^{-2\tau} \min_{x\in\overline{\Omega}}\vinit(x),\ 
  c_0m(1-e^{-\tau})
  \right\}
  \right).
\end{align} 
\end{lem}
\begin{proof}
For a function $f:\Omega \times [0,T)\to \mathbb{R}$ 
putting $\widetilde{f}(x,t):=f(x,\lambda t)$ for $(x,t)\in \Omega \times (0,\frac{T}{\lambda})$, 
we see that $\widetilde{\vtau}$ satisfies 
\begin{align*}
\begin{cases}
  (\widetilde{\vtau})_t=\Delta \widetilde{\vtau} - \widetilde{\vtau} + \widetilde{u}, 
&
  x\in\Omega,\ t\in (0,\frac{T}{\lambda}),
\\
  \nabla \widetilde{\vtau}\cdot \nu =0, 
&
  x\in\partial\Omega,\ t\in(0,\frac{T}{\lambda}), 
\\
  \widetilde{\vtau}(x,0)=\vinit(x),
&
  x\in \Omega. 
\end{cases}
\end{align*}
Thus since the mass conservation yields that $\into u(\cdot,\lambda t)=m$ 
for all $t\in (0,\frac T\lambda)$ and any $\lambda>0$, 
we infer from \cite{Fujie_2015, Fujie_DocTh} 
(see also \cite[Lemma 2.1]{Mizukami-Yokota_02}) that 
$\widetilde{\eta}$ defined as \eqref{defetatilde} satisfies
\begin{align*}
\inf_{x\in \Omega}\widetilde{\vtau}(x,t)\ge \widetilde{\eta} 
\end{align*} 
for all $t\in (0,\frac{T}{\lambda}$), 
which implies this lemma. 
\end{proof}

We next recall the result which is concerned with local existence of 
solutions (see e.g., \cite[Lemma 2.1]{Winkler_2010ab}). 

%
%
%
%
\begin{lem}\label{lem;localexistence}
Let $\Omega$ be a bounded domain in $\Rn$ $(n\ge 2)$ 
with smooth boundary, 
and assume that \eqref{mainass} and \eqref{ini} are 
satisfied. 
Then for all $\lambda>0$ 
there exists $\tmax\in (0,\infty]$ such that 
the problem \eqref{cp} possesses a unique solution 
$(\utau,\vtau)$ fulfilling 
\begin{align*}
  &\utau,\vtau\in C(\overline{\Omega}\times [0,\tmax))
  \cap C^{2,1}(\overline{\Omega}\times (0,\tmax)),
\\
  &\utau (x,t)\ge 0\quad \mbox{for all}\ x\in \Omega\ \mbox{and all} \ t>0,
\\
  &\into \utau\cd = \into \uinit
\quad \mbox{and}\quad
  \inf_{x\in \Omega}\vtau (x,t) \ge \eta
  \quad \mbox{for all}\ t\in (0,\tmax) \ \mbox{and}\ \lambda>0.
\end{align*}
Moreover, either $\tmax=\infty$ or 
\begin{align*}
  \limsup_{t\to\tmax}(\lp{\infty}{\utau\cd}
  + \|\vtau\cd\|_{W^{1,q}(\Omega)})=\infty. 
\end{align*}
\end{lem}
%

%
%
%
%
Then we shall provide the following two lemmas which hold keys to derive 
important estimates for the proofs of main results. 
\begin{lem}\label{lem;estiforvandnablav}
Let $1\le \theta,\mu\le \infty$. 
Then we have the following properties. 
\begin{itemize}
\item[{\rm (i)}]
If $\frac{n}{2}(\frac{1}{\theta}-\frac{1}{\mu})<1$, 
then there exists $C=C(\theta,\mu)>0$ such that 
\begin{align}\label{esti;v;fromu}
  \lp{\mu}{\vtau\cd}\le 
  C \left(
  1+ \sup_{s\in (0,\tmax)}\lp{\theta}{\utau(\cdot,s)}
  \right)
\end{align}
for all $t\in (0,\tmax)$ and any $\lambda>0$.
\item[{\rm (ii)}]
If $\frac{1}{2}+
\frac{n}{2}(\frac{1}{\theta}-\frac{1}{\mu})<1$, 
then there exists $C=C(\theta,\mu)>0$ such that 
\begin{align}\label{esti;nablav;fromu}
  \lp{\mu}{\nabla\vtau\cd}\le 
  C \left(
  1+ \sup_{s\in (0,\tmax)}\lp{\theta}{\utau(\cdot,s)}
  \right)
\end{align}
for all $t\in (0,\tmax)$ and any $\lambda>0$.
\end{itemize}
\end{lem}
\begin{proof}
In the case that $\frac{n}{2}(\frac{1}{\theta}-\frac{1}{\mu})<1$ 
by using the transformation 
$\widetilde\vtau(x,t):=\vtau (x,\lambda t)$ 
for $(x,t)\in \Omega \times (0,\frac{\tmax}{\lambda})$ 
and a straightforward 
application of well-known smoothing estimates for the heat semigroup under  homogeneous  Neumann boundary conditions (see \cite[Lemma 2.4 (i)]{Winkler_2011})  we have that 
\begin{align*}
\lp{\mu}{\widetilde{\vtau}\cd}\le C_1\left(
1+\sup_{s\in \big(0,\frac{\tmax}{\lambda}\big)}\lp{\theta}{\utau(\cdot,\lambda s)}
\right)
\end{align*}
for all $t\in (0,\frac{\tmax}{\lambda})$ and any $\lambda>0$ 
with some $C_1=C_1(\theta,\mu)>0$, 
which implies that \eqref{esti;v;fromu} holds. 
Similarly, in the case that $\frac{1}{2}+\frac{n}{2}(\frac{1}{\theta}-\frac{1}{\mu})<1$
the same transformation and \cite[Lemma 2.4 (ii)]{Winkler_2011} 
derive \eqref{esti;nablav;fromu} 
with some $C_2=C_2(\theta,\mu)>0$ independent of $\lambda>0$. 
\end{proof}

%
%
%
%
\begin{lem}\label{lem;estimateforufromp}
Let $\lambda>0$. 
If there exist $p>\frac{n}{2}$ and $M>0$ 
such that  
\begin{align*}
\lp{p}{\utau\cd} \le M
\quad \mbox{for all}\ t\in (0,\tmax), 
\end{align*}
then there exists $C=C(p,M)>0$ such that 
\begin{align*}
\lp{\infty}{\utau\cd}\le C
\quad \mbox{for all} \ t\in (0,\tmax). 
\end{align*}
Moreover, if $p$ and $M$ are independent of $\lambda\in (0,\lambda_0)$ 
with some $\lambdaz \in (0,1)$, 
then $C$ is also independent of $\lambda\in (0,\lambda_0)$. 
\end{lem}
\begin{proof}
Thanks to assumption, there exist $p>\frac{n}{2}$ and $C_1>0$ such that 
\begin{align}\label{ineq;Lpforu;proof}
  \lp{p}{\utau\cd}\le C_1
  \quad \mbox{for all}\ t\in(0,\tmax). 
\end{align}
Then we can find $r=r(p)\ge 1$ and $\mu=\mu(p)\ge 1$ such that 
\begin{align*}
  n < r < \mu < \frac{np}{(n-p)_+}, 
\end{align*} 
because $p>\frac n2$. 
Therefore Lemma \ref{lem;estiforvandnablav} and \eqref{ineq;Lpforu;proof} 
enable us to obtain that 
\begin{align*}
 \lp{\mu}{\nabla \vtau\cd} \le C_2 
 \quad\mbox{for all}\ t\in (0,\tmax)
\end{align*}
with some $C_2=C_2(p,C_1)>0$. 
Now we put  
\[ 
  A(T')=\sup_{t\in (0,T')}\lp{\infty}{\utau\cd}  <\infty
\] 
for $T'\in (0,\tmax)$ and will show that $A(T')\le C$ for all $T'\in (0,\tmax)$ with some $C>0$. 
In order to obtain the estimate for $A(T')$ 
we set $t_0:=(t-1)_+$ for $t\in (0,T')$ and 
represent $\utau$ according to 
\begin{align}\label{semi;utau}
  \utau(\cdot,t)
  &=e^{(t-t_0)\Delta}\utau(\cdot,t_0) 
  - \int_{t_0}^t e^{(t-s)\Delta}\nabla\cdot 
  (\utau(\cdot,s)\chi(\vtau(\cdot,s))\nabla \vtau(\cdot,s))\,ds 
  \\\notag
  &=:I_1(\cdot,t)+I_2(\cdot,t).
\end{align}
In the case that $t\le 1$ 
the order preserving property of 
the Neumann heat semigroup implies that 
\begin{align}\label{estiI1-1}
  \lp{\infty}{I_1(\cdot,t)}
  \le \lp{\infty}{\uinit}
  \quad 
  \mbox{for all}\ t\in(0,T'). 
\end{align}
In the case that $t>1$ 
by the $L^p$-$L^q$ estimates for 
$(e^{\tau \Delta})_{\tau\ge 0}$ 
(see \cite[Lemma 1.3 (i)]{win_aggregationvs}) 
and \eqref{ineq;Lpforu;proof} we can see that there exists $C_3=C_3(p)>0$ such that 
\begin{align}\label{estiI1-2}
  \lp{\infty}{I_1(\cdot,t)}
  \le C_3\lp{p}{\utau(\cdot,t_0)} 
  \le C_3C_1
  \quad 
  \mbox{for all}\ t\in (0,T'). 
\end{align}
On the other hand, noting that
\[
  \lp{r}{\utau\cd\nabla \vtau\cd}\le 
  \lp{\frac{r\mu}{\mu-r}}{\utau\cd}\lp{\mu}{\nabla \vtau\cd}
  \le  A(T')^{1-\frac{\mu-r}{r\mu}}\lp{1}{\uinit}^{\frac{\mu-r}{r\mu}} C_2
\]
holds for all $t\in(0,T')$, 
we obtain from a known smoothing property of 
$(e^{\tau\Delta})_{\tau\ge 0}$ (see \cite[Lemma 3.3]{FIWY_2016}) 
that 
\begin{align*}
  \lp{\infty}{I_2\cd} 
  &\le 
  \frac{\chi_0}{(a+\eta)^k}\int_{t_0}^t (t-s)^{-\frac{1}{2}-\frac{n}{2r}}\lp{r}{\utau(\cdot,s)\nabla \vtau(\cdot,s)}\,ds
\\
  &\le \frac{\chi_0 A(T')^{1-\frac{\mu-r}{r\mu}}\lp{1}{\uinit}^{\frac{\mu-r}{r\mu}} C_2}{(a+\eta)^k} 
  \int_0^1  \sigma^{-\frac{1}{2}-\frac{n}{2r}}\,d\sigma 
\end{align*} 
for all $t\in(0,T')$. Since $\int_0^1  \sigma^{-\frac{1}{2}-\frac{n}{2r}}\,d\sigma$ is finite 
from $\frac{n}{2r}<\frac{1}{2}$, 
there exists 
$C_4=C_4(p,C_1)>0$ such that 
\begin{align}\label{estiI2}
\lp{\infty}{I_2\cd}\le C_4 A(T')^{1-\frac{\mu-r}{r\mu}} 
\quad \mbox{for all}\ t\in(0,T').
\end{align}
Therefore a combination of \eqref{estiI1-1} and \eqref{estiI1-2}, 
along with \eqref{estiI2} 
derives that 
\begin{align*}
  \lp{\infty}{\utau(\cdot,t)}
  \le \lp{\infty}{I_1(\cdot,t)}+\lp{\infty}{I_2(\cdot,t)}
  \le C_5 + C_4A(T')^{1-\frac{\mu-r}{r\mu}}
\end{align*}
holds for all $t\in (0,T')$ with some $C_5=C_5(p,C_1)>0$,  
which together with $\frac{\mu-r}{r\mu}<1$ 
implies that there exists $C_6=C_6(p,C_1)>0$ such that 
\begin{align*}
 A(T')=\sup_{t\in(0,T')}
 \lp{\infty}{n(\cdot,t)}\le C_6 
 \quad 
 \mbox{for all}\ T'\in (0,\tmax).
\end{align*}
Therefore we can attain the $L^\infty$-estimate for $\utau$. 
%
Moreover, in the case that $p$ and $C_1$ are independent of $\lambda\in(0,\lambda_0)$ 
with some $\lambda_0\in(0,1)$, aided by Lemma \ref{lem;estiforvandnablav}, 
we can see that the constants appearing in this  
proof are independent of $\lambda\in(0,\lambda_0)$. 
\end{proof}

%
%

\section{Global existence}

In this section we will show global existence and 
boundedness in \eqref{cp} (Theorem \ref{mainth}) 
and the uniform-in-$\lambda$ estimate for the solution 
(Corollary \ref{mainth2}). 
Thanks to Lemma \ref{lem;estimateforufromp}, 
our aim is to derive the $L^p$-estimate for $\utau$ with $p>\frac n2$. 
We first prove the following lemma 
which plays an important role in the proofs of the main results. 

%
%
%
%
\begin{lem}\label{lem;ftau}
For all $\ep\in (0,\frac{1}{2})$, $p>1$ 
and $\lambda\ge 0$, 
\begin{align*}
 p\lambda^2 + p - 2p\lambda + 4\ep p\lambda 
 + 4\lambda -4\ep\lambda > 0
\end{align*}
holds. 
\end{lem}
\begin{proof}
Let $\ep\in (0,\frac{1}{2})$. 
We shall see that the discriminant of $f_p(\lambda):=p\lambda^2 + (-2p+4\ep p+4-4\ep)\lambda+p$ is 
negative: 
\begin{align}\label{purpose;discriminant}
  D_\lambda := 4(-(1-(2\ep-1)^2)p^2 -4(1-\ep)(1-2\ep)p + 4(1-\ep)^2)<0 
  \quad \mbox{for all} \ p>1. 
\end{align}
Now we put $g(p):=-(1-(2\ep-1)^2)p^2 -4(1-\ep)(1-2\ep)p + 4(1-\ep)^2$. 
Then since 
\[
g'(p)=-2(1-(2\ep-1)^2)p-4(1-\ep)(1-2\ep)<0 
\]
from $\ep\in(0,\frac 12)$ and $g(1)=0$ hold, we have 
$g(p)<0$ for all $p>1$, which means that 
\eqref{purpose;discriminant} holds. 
Therefore we obtain 
\[
  f_p(\lambda)>0 \quad \mbox{for all}\ \lambda\ge 0 \ \mbox{and all} \ p>1, 
\]
which entails this lemma. 
\end{proof}

In order to establish the $L^p$-estimate for $\utau$ with $p>\frac{n}{2}$ 
we put 
\[
  \vphi (s):= 
  \exp\left\{
    -r\int_\eta^s \frac{1}{(a+\sigma)^k}\,d\sigma
  \right\} 
  \ \mbox{for} \ s\ge \eta  
\]
with some $r>0$ and will show the following two lemmas which 
derive a differential inequality for $\into \utau^p \vphi(\vtau)$. 

%
%
%
%
\begin{lem}\label{lem;difineq1}
Assume that \eqref{mainass} is satisfied 
with some $\vtheta\in (0,1)$, $a\ge0$, $k>1$ and $\chi_0>0$ satisfying  
\begin{align}\label{ass;P-P}
  \frac{\left(
  (1-\lambda+2\lambda\ep)_+ p + \sqrt{p(p\lambda^2 + p - 2p\lambda + 4\ep p\lambda + 4\lambda -4\ep\lambda)}\right)\chi_0}{2\lambda(1-\ep)} 
  - \frac{k}{\lambda} (a+\eta)^k 
  \le 0
\end{align}
with some $\ep\in (0,\frac{1}{2})$ and $p>1$. 
Then 
\begin{align}\label{ineq;difuphi}
 &\frac d{dt}\into \utau^p\vphi(\vtau) 
 \le 
-  \ep p(p-1)\into \utau^{p-2}\vphi(\vtau)
 |\nabla \utau|^2  
 + \frac r\lambda \into \utau^p\vphi(\vtau)
 \frac{\vtau}{(a+\vtau)^k}
\end{align}
holds, where 
\begin{align}\label{def;r}
r:=\lambda (p-1)\chi_0\sqrt{\frac{p}{p\lambda^2 -2p \lambda +p +4p\ep \lambda+4\lambda-4\ep \lambda}}.
\end{align}
\end{lem}
\begin{proof}
In light of integration by parts and the Young inequality, 
we obtain from straightforward calculations that 
\begin{align*}
   \frac{d}{dt}\into \utau^p \vphi(\vtau) &= 
   -p(p-1)\into \utau^{p-2}\vphi(\vtau)|\nabla \utau|^2
 \\ &\quad\,
   + \into \utau^{p-1} \left(p(p-1)\chi(\vtau)\vphi(\vtau)
   -\left(1+\frac{1}{\lambda}\right)p\vphi'(\vtau)\right)\nabla \utau\cdot \nabla\vtau
   \\
   &\quad\, 
   +\into \utau^p \left( p\chi(\vtau)\vphi'(\vtau)-\frac{1}{\lambda}\vphi''(\vtau) \right)|\nabla \vtau|^2
  + \frac{1}{\lambda}\into \utau^p \vphi'(\vtau)(\utau-\vtau)
\\
  &\le -\ep p(p-1)\into \utau^{p-2}|\nabla \utau|^2 
  + \into \utau^p\vphi(\vtau)H_{r,\ep}(\vtau)|\nabla \vtau|^2
  \\&\quad\,
  + \frac{r}{\lambda}\into \utau^p \vphi(\vtau)\frac{\vtau}{(a+\vtau)^k} 
\end{align*}
holds for all $\ep\in (0,\frac{1}{2})$, where  
\begin{align*}
  H_{r,\ep}(s) &:=  
  \frac{p\lambda^2 -2p \lambda +p +4p\ep \lambda+4\lambda-4\ep \lambda}
  {4\lambda^2(1-\ep)(p-1)(a+s)^{2k}}r^2
\\ &\quad\
  + \left(
    \frac{(1-\lambda+2\lambda\ep)_+p\chi_0}{2\lambda(1-\ep)(a+s)^{2k}} 
    - \frac{k}{\lambda(a+s)^{k+1}}
  \right)r
  + \frac{p(p-1)\chi_0^2}{4(1-\ep)(a+s)^{2k}}. 
\end{align*}
Now since \eqref{ass;P-P} holds with 
some $\ep\in(0,\frac{1}{2})$ and $p>1$, an argument similar to 
that in the proof of \cite[Lemma 4.1]{Mizukami-Yokota_02} implies that 
\begin{align*}
  H_{r,\ep} (s) \le 0 \quad 
  \mbox{for all}\ s\ge \eta 
\end{align*}
with $r>0$ defined as \eqref{def;r}, 
which leads to the end of the proof. 
\end{proof}

%
%
%
%
\begin{lem}\label{lem;difineq2}
Assume that \eqref{mainass} and \eqref{ass;P-P} 
are satisfied 
with $\ep\in (0,\frac{1}{2})$ and $p>1$. 
Then there exist $b>1$ and $c_1,c_2,c_3>0$ such that 
\begin{align*}
\frac{d}{dt}\into \utau^p \vphi(\vtau)
\le 
-c_1\left( \into \utau^p \vphi(\vtau)\right)^b
+ c_2 \into \utau^p \vphi(\vtau) 
+c_3.
\end{align*}
Moreover, if there are $\lambda_0>0$, $p>1$ and $\ep\in(0,\frac{1}{2})$ such that 
\eqref{ass;P-P} holds for all $\lambda\in (0,\lambda_0)$, then the constants 
$b, c_1,c_2,c_3$ are independent of $\lambda\in (0,\lambda_0)$. 
\end{lem}
\begin{proof}
By virtue of Lemma \ref{lem;difineq1}, 
we have that \eqref{ineq;difuphi} 
holds with $r>0$ defined as \eqref{def;r}. 
We first obtain from the boundedness of the function 
$s\mapsto \frac{s}{(a+s)^k}$ on $[\eta,\infty)$ $(k>1)$  
that there is a constant $C_1>0$ 
satisfying  
\begin{align}\label{inequ;boundednessofs}
  \into \utau^p \vphi(\vtau)\frac{\vtau}{(a+\vtau)^k} 
  \le C_1 \into \utau^p \vphi(\vtau). 
\end{align}
Then the fact 
\begin{align}\label{ineq;phiul}
  \exp\left\{
    \frac{-r}{(k-1)(a+\eta)^{k-1}}
  \right\} 
  \le \vphi(s) \le 1 
  \quad \mbox{for all}\ s\ge \eta
\end{align}
and the Gagliardo--Nirenberg inequality 
\begin{align*}
  \lp{2}{\utau^\frac{p}{2}} 
  \le C_2 \left(
  \lp{2}{\nabla \utau^{\frac{p}{2}}}+\lp{\frac{2}{p}}{\utau^\frac{p}{2}}
  \right)^\alpha 
  \lp{\frac{2}{p}}{\utau^{\frac{p}{2}}}^{(1-\alpha)}
\end{align*}
with $\alpha:=\frac{\frac{pn}{2}-\frac{n}{2}}{\frac{pn}{2}+1-\frac{n}{2}}\in (0,1)$ 
and some constant $C_2>0$ imply that 
there exists $C_3>0$ such that 
\begin{align}\label{esti;nablau}
  \into \utau^p\vphi(\vtau) 
  \le C_3 \left(
  \exp\left\{\frac{r}{(k-1)(a+\eta)^{k-1}}\right\}
  \into \utau^{p-2}\vphi(\vtau)|\nabla \utau|^2+1
  \right)^\alpha. 
\end{align}
Therefore a combination of 
\eqref{ineq;difuphi}, \eqref{inequ;boundednessofs} and \eqref{esti;nablau} yields 
\begin{align}\label{ineq;difuphi2}\notag
  \frac d{dt}\into \utau^p\vphi(\vtau) 
  &\le  
  -\ep p(p-1)C_3^{-\frac 1\alpha} \exp\left\{\frac{-r}{(k-1)(a+\eta)^{k-1}}\right\}
  \left(\into \utau^p\vphi(\vtau)\right)^{\frac 1\alpha} 
\\&\quad \,
  + \frac {rC_1}{\lambda} \into \utau^p\vphi(\vtau)  
  + \exp\left\{\frac{-r}{(k-1)(a+\eta)^{k-1}}\right\}. 
\end{align}
Moreover, if there are $\lambda_0>0$, $p>1$ and $\ep\in(0,\frac{1}{2})$ such that 
\eqref{ass;P-P} holds for all $\lambda\in (0,\lambda_0)$, then noting from 
the definition of $r$ (see \eqref{def;r}) and the existence of $C_4>0$ 
satisfying 
$
  p\lambda^2 -2p \lambda +p +4p\ep \lambda+4\lambda-4\ep \lambda \ge C_4
$ 
for all $\lambda\in [0,\lambda_0]$ (from Lemma \ref{lem;ftau}) that 
\begin{align}\label{esti;unifor;phi}
  &1\le \exp\left\{\frac{r}{(k-1)(a+\eta)^{k-1}}\right\}\le 
  \exp\left\{  
  \frac{\lambda_0 (p-1)\chi_0}{(k-1)(a+\eta)^{k-1}}\sqrt{\frac{p}{C_4}} 
  \right\}=:C_5
\end{align}
and
\begin{align*}
  \frac{rC_1}{\lambda} \le C_1
  (p-1)\chi_0\sqrt{\frac{p}{C_4}} =:C_6,  
\end{align*}
we can see from \eqref{ineq;difuphi2} that 
\begin{align*}
  \frac d{dt}\into \utau^p\vphi(\vtau) 
  &\le  
  -\ep p(p-1)C_3^{-\frac 1\alpha} C_5^{-1}
  \left(\into \utau^p\vphi(\vtau)\right)^{\frac 1\alpha} 
  + C_6 \into \utau^p\vphi(\vtau)  
  + 1. 
\end{align*}
Thus we can show this lemma. 
\end{proof}

Now we have already provided all tools to establish 
the $L^p$-estimate for $\utau$ 
under the condition \eqref{ass;P-P}. 
%
%
%
%

\begin{lem}\label{lem;estimateforu}
Assume that \eqref{mainass} and \eqref{ass;P-P} 
are satisfied 
with some $\ep\in(0,\frac{1}{2})$ and $p>1$. 
Then there exists $C>0$ such that 
\begin{align*}
 \lp{p}{\utau\cd}\le C \quad\mbox{for all}\ t\in (0,\tmax). 
\end{align*}
Moreover, if there are $\lambda_0>0$, $p>1$ and $\ep\in(0,\frac{1}{2})$ such that 
\eqref{ass;P-P} holds for all $\lambda\in (0,\lambda_0)$, 
then $C$ is independent of $\lambda\in (0,\lambda_0)$. 
\end{lem}
\begin{proof}
Since \eqref{ineq;phiul} holds, 
Lemma \ref{lem;difineq2} and the standard ODE comparison argument 
lead to the $L^p$-estimate for $\utau$. 
Moreover, if there are $\lambda_0>0$, $p>1$ and $\ep\in(0,\frac{1}{2})$ such that 
\eqref{ass;P-P} holds for all $\lambda\in (0,\lambda_0)$, 
then a combination of Lemma \ref{lem;difineq2} and \eqref{ineq;phiul}, 
along with \eqref{esti;unifor;phi} implies the desired uniform-in-$\lambda$ 
estimate. 
\end{proof}

Now we are ready to attain the $L^p$-estimate for $\utau$ 
under the condition \eqref{mainass2} or \eqref{mainass3}. 
Here we note that 
\eqref{mainass2} is the case of \eqref{ass;P-P} with $p=\frac n2$ and $\ep=0$, 
and \eqref{mainass3} is the case of  \eqref{ass;P-P} 
with $p=\frac n2$, $\ep=0$ and $\lambda=0$. 
%
%
%
%
\begin{corollary}\label{lem;P-PLp}
 Assume that \eqref{mainass} and \eqref{mainass2} 
 are satisfied. 
 Then there exist $p>\frac{n}{2}$ and $C>0$ 
 such that 
 \begin{align*}
   \lp{p}{\utau\cd}\le C
   \quad\mbox{for all}\ t\in (0,\tmax). 
 \end{align*}
\end{corollary}
\begin{proof}
Invoking to \eqref{mainass2}, 
we obtain from the continuity argument that there are 
$p>\frac n2$ and $\ep\in (0,\frac 12)$ such that \eqref{ass;P-P} holds. 
Thus Lemma \ref{lem;estimateforu} enables us to see 
the $L^p$-estimate for $\utau$. 
\end{proof}
%
%
%
%
\begin{corollary}\label{lem;P-ELp}
 Assume that \eqref{mainass} and \eqref{mainass3} 
 are satisfied. 
 Then there exist $\lambda_0\in (0,1)$, $p>\frac{n}{2}$ 
 and $C>0$ such that 
 \begin{align*}
 \lp{p}{\utau\cd}\le C
 \end{align*}
for all $t\in (0,\tmax)$ and any $\lambda\in (0,\lambda_0)$. 
\end{corollary}
\begin{proof}
In light of \eqref{mainass3}, we can find 
$\lambda_0\in(0,1)$, $p>\frac n2$ and $\ep\in(0,\frac 12)$ such that 
\eqref{ass;P-P} holds for all $\lambda\in (0,\lambda_0)$.  
Therefore from Lemma \ref{lem;estimateforu} we obtain this lemma. 
\end{proof}

A combination of 
Lemma \ref{lem;estiforvandnablav} and Corollaries \ref{lem;P-PLp}, \ref{lem;P-ELp} 
implies the following lemma. 
\begin{lem}\label{lem;W1qesti;v}
Assume that \eqref{mainass} and \eqref{mainass2}, 
or \eqref{mainass} and \eqref{mainass3} 
are satisfied. 
Then there exists $C>0$ such that 
\begin{align*}
  \|\vtau\cd\|_{W^{1,q}(\Omega)}\le C 
  \quad \mbox{for all}\ t\in (0,\tmax).
\end{align*} 
Moreover, if \eqref{mainass} and \eqref{mainass3} are satisfied, 
then $C$ is independent of $\lambda\in (0,\lambda_0)$ with some $\lambda_0\in (0,1)$. 
\end{lem}

\begin{proof}[{\rm \bf Proof of Theorem \ref{mainth}}]
Lemmas \ref{lem;estimateforufromp}, \ref{lem;W1qesti;v} and 
Corollary \ref{lem;P-PLp} derive Theorem \ref{mainth}. 
\end{proof}

\begin{proof}[{\rm \bf Proof of Corollary \ref{mainth2}}]
Lemmas \ref{lem;estimateforufromp}, \ref{lem;W1qesti;v} 
and Corollary \ref{lem;P-ELp}  
lead to Corollary \ref{mainth2}. 
\end{proof}

%
%

\section{Convergence}

In this section we will show that 
solutions of \eqref{cp} 
converge to those of \eqref{cpp-e} (Theorem \ref{mainth3}). 
Here we note that Arguments in this section are based on those in 
the proof of 
\cite[Theorem 1.1]{Wang-Winkler-Xiang}; 
thus I shall only show brief proofs. 
Here we assume that \eqref{mainass} and \eqref{mainass3} are satisfied. 
Then thanks to Corollary \ref{mainth2}, there exists $\lambda_0\in (0,1)$ such that 
for all $\lambda\in (0,\lambda_0)$ the problem \eqref{cp} possesses a unique global classical 
solution $(\utau,\vtau)$ satisfying the uniform-in-$\lambda$ estimate 
\eqref{estimate;uniformintau}. 
We first confirm the following lemma 
which is a cornerstone of this work. 
\begin{lem}\label{lem;boundedinholder}
For all sequences of numbers $\{\lambdan\}_{n\in\mathbb{N}}\subset (0,\lambda_0)$ 
satisfying $\lambdan \searrow 0$ as $n\to\infty$ 
there exist a subsequence $\lambdanj \searrow 0$ and functions 
\begin{align*}
  u\in C(\obar\times [0,\infty))\cap 
  C^{2,1}(\obar\times (0,\infty)) 
\ \mbox{and} \ 
  v \in C^{2,0}(\obar\times (0,\infty))\cap 
  L^\infty([0,\infty);W^{1,q}(\Omega)) 
\end{align*}
such that for all $T>0$, 
\begin{align*}
 &\utaunj \to u \quad 
 \mbox{in}\ C_{\rm loc}(\obar\times [0,\infty)), 
\\
 &\vtaunj \to v \quad \mbox{in}\ 
 C_{\rm loc}(\obar\times (0,\infty)) \cap L^2_{\rm loc}((0,\infty);W^{1,2}(\Omega))
\end{align*}
as $j\to \infty$. 
Moreover, $(u,v)$ solves \eqref{cpp-e} 
classically. 
\end{lem}
\begin{remark}
This lemma also implies that 
if $\chi$ satisfies \eqref{mainass} and \eqref{mainass3} then 
global existence and boundedness in \eqref{cpp-e} hold. 
\end{remark}
\begin{proof}
From the assumption in this section 
and the standard parabolic regularity argument \cite[Theorem 1.3]{Porzio-Vespri}
we see that $\{\utau\}_{\itau}$ is bounded in 
$C^{\alpha,\frac{\alpha}{2}}_{\rm loc}(\ol{\Omega}\times[0,\infty))$ 
with some $\alpha\in (0,1)$. 
Thus the Arzel\`{a}--Ascoli theorem 
and the boundedness of $\|\nabla \vtau\|_{L^\infty(0,\infty;W^{1,q}(\Omega))}$ yields that we can find a subsequence 
$\lambdanj \searrow 0$ and functions 
\[
  u\in C^{\alpha,\frac{\alpha}{2}}_{\rm loc}
    (\ol{\Omega}\times [0,\infty)) 
\quad \mbox{and} \quad 
  v\in L^\infty (0,\infty;W^{1,q}(\Omega)) 
\]
satisfying 
\begin{align*}
 \utaunj \to u \quad 
 \mbox{in}\ C_{\rm loc}(\obar\times [0,\infty)) 
\quad \mbox{and} \quad 
 \vtaunj \overset{\ast}{\rightharpoonup} v \quad \mbox{in}\ 
 L^\infty (0,\infty;W^{1,q}(\Omega)) 
\end{align*}
as $j\to \infty$. 
Then arguments similar to those in 
the proof of \cite[Theorem 1.1]{Wang-Winkler-Xiang} 
enable us to attain this lemma.  
\end{proof}

%
%
%
%
We next verify the following lemma which implies that 
the pair of functions $(u,v)$ provided 
by Lemma \ref{lem;boundedinholder} is independent of a choice of 
a sequence $\lambda_n\searrow 0$. 

\begin{lem}\label{uniqueness}
A solution $(\uobar,\vobar)$ of \eqref{cpp-e} satisfying 
\[
  \uobar\in C(\obar\times [0,\infty))\cap 
  C^{2,1}(\obar\times (0,\infty)) 
\ \mbox{and} \ 
  \vobar \in C^{2,0}(\obar\times (0,\infty))\cap 
  L^\infty([0,\infty);W^{1,q}(\Omega)) 
\]
is unique. 
\end{lem}
\begin{proof}
Let $({\uobar}_1,{\vobar}_1)$ and $({\uobar}_2,{\vobar}_2)$ be solutions to \eqref{cpp-e}  
and put $y(x,t):= \uobar_1(x,t)-\uobar_2(x,t)$ for $(x,t)\in \Omega \times (0,\infty)$.  
Then an argument similar to that in 
the proof of \cite[Lemma 2.1]{stinner_tello_winkler} 
implies that $y(x,t)=0$. Thus we can obtain  this lemma. 
\end{proof}

%
%
%
%
%
%
Finally we shall establish convergence of 
the solution $(\utau,\vtau)$ for \eqref{cp} as $\lambda\searrow 0$. 
\begin{lem}\label{lem;convergence;final}
The solution $(\utau,\vtau)$ of \eqref{cp} with $\lambda\in (0,\lambda_0)$ 
satisfies that for all $T>0$ 
\begin{align*}
 &\utau \to u \quad 
 \mbox{in}\ C_{\rm loc}(\obar\times [0,\infty)), 
\\
 &\vtau \to v \quad \mbox{in}\ 
 C_{\rm loc}(\obar\times (0,\infty)) \cap L^2_{\rm loc}((0,\infty);W^{1,2}(\Omega)) 
\end{align*}
as $\lambda\searrow 0$, where $(u,v)$ is the solution of \eqref{cpp-e} 
provided by Lemma \ref{lem;boundedinholder}. 
\end{lem}
\begin{proof}
Lemmas \ref{lem;boundedinholder} and \ref{uniqueness} yield that 
there exists the pair of the functions $(u,v)$ such that 
for any sequences $\{\lambda_n\}_{n\in\mathbb{N}}\subset (0,\lambda_0)$ 
satisfying $\lambda_n\searrow 0$ as $n\to \infty$ 
there is a subsequence $\lambda_{n_j} \searrow 0$ 
such that for all $T>0$, 
\begin{align*}
 &\utaunj \to u \quad 
 \mbox{in}\ C_{\rm loc}(\obar\times [0,\infty)), 
\\
 &\vtaunj \to v \quad \mbox{in}\ 
 C_{\rm loc}(\obar\times (0,\infty)) \cap L^2_{\rm loc}((0,\infty);W^{1,2}(\Omega))
\end{align*}
as $j\to \infty$, which enables us to see this lemma. 
\end{proof}

\begin{proof}[{\rm \bf Proof of Theorem \ref{mainth3}}]
Lemma \ref{lem;convergence;final} directly shows Theorem \ref{mainth3}. 
\end{proof}
%
\smallskip
\section*{Acknowledgments}
The author would like to thank Professor Michael Winkler 
for pointing out mistakes in 
the previous proof of Theorem 1.3, 
and moreover, 
appreciated his modifying arguments in his joint paper 
with Professors Yulan Wang and Zhaoyin Xiang 
(\cite{Wang-Winkler-Xiang}).



\end{document}